\documentclass{article}
\usepackage{authblk}

\usepackage{amsmath}
\usepackage{amsthm}
\usepackage{amsfonts}
\usepackage[all]{xy}
\usepackage{url}

\usepackage{graphicx}

\usepackage{color}
\usepackage{xcolor}

\DeclareMathSymbol{\rightrightarrows}  {\mathrel}{AMSa}{"13}

\def\sq{\operatorname{sq}}

\catcode`\@=11
\def\varholim@#1#2{\mathop{\vtop{\ialign{##\crcr
 \hfil$#1\m@th\operator@font holim$\hfil\crcr
 \noalign{\nointerlineskip\kern\ex@}#2#1\crcr
 \noalign{\nointerlineskip\kern-\ex@}\crcr}}}}
\def\hocolim{\mathpalette\varholim@\rightarrowfill@} 
\def\hoinvlim{\mathpalette\varholim@\leftarrowfill@}
\catcode`\@=\active

\newtheorem{theorem}{Theorem}
\newtheorem{lemma}[theorem]{Lemma}
\newtheorem{corollary}[theorem]{Corollary}

\newtheorem{proposition}[theorem]{Proposition}

\theoremstyle{definition}

\newtheorem{ex}[theorem]{Example}

\newtheorem{rem}[theorem]{Remark}

\begin{document}

\title{\bf Stable components and layers}
\author{J.F. Jardine\thanks{Supported by NSERC.}}

\affil{\small Department of Mathematics\\University of Western Ontario
}
\affil{jardine@uwo.ca}


\maketitle

\begin{abstract}
\noindent
Component graphs $\Gamma_{0}(F)$  are defined for arrays of sets $F$, and in particular for arrays of path components for Vietoris-Rips complexes and Lesnick complexes. The path components of
$\Gamma_{0}(F)$ are the {\it stable components} of the array $F$.
The stable components for the system of Lesnick complexes $\{ L_{s,k}(X) \}$ for a finite data set $X$ decompose into {\it layers}, which are themselves path components of a graph. Combinatorial scoring functions are defined for layers and stable components.
\end{abstract}

\smallskip


\section*{Introduction}

Astronomers say that a cluster is a ``group of stars or galaxies forming a relatively close association''.
\begin{center}
\includegraphics[width=100pt]{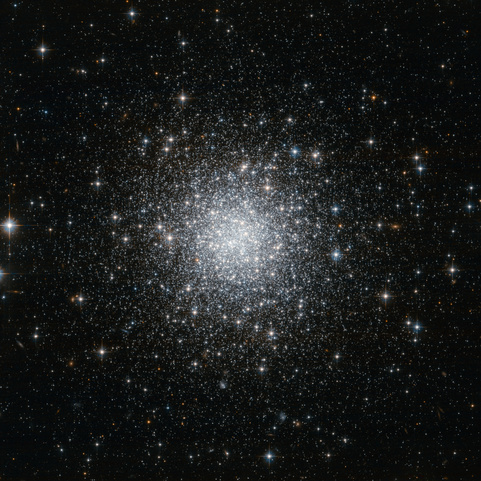}
\end{center}
Clusters are distinguished by relative density: they are concentrated collections of objects, surrounded by voids.

In the late 1700s, the brother and sister team William and Caroline Herschel found and classified ``stellar over-densities'' by counting stars in grids of regions of space. The method used today follows the same principle, although it is done with sophisticated imaging equipment coupled with computer analysis that filters out light artifacts.
Relative to big picture items such as the cosmic microwave background, clusters are anomalies --- they are small, dense collections of stellar objects.

The same sort of big picture/little picture dichotomy is present for other large data sets. In financial data, a dense, relatively small collection of rapid low-value transactions could point to an instance of money laundering, while a large scale analysis can detect sectoral or global market fluctuations.
\medskip

Colloquially, clusters are collections of data points in relative close proximity, within some space that is defined by a finite list of parameters.
We shall assume, more precisely, that a data set (or a data cloud) $X$ is a finite collection of points inside real vector space $\mathbb{R}^{n}$. The motivating idea behind clustering is to find regions in the ambient space that contain dense populations of elements of the data set $X$. This is usually done in data analysis by applying partitioning methods to the data set $X$, or by studying hierarchies of such partitions.

The definitional structure and methods of this paper represent a departure from traditional clustering, although partitions and hierarchies (really, trees) of partitions are both used. The central objects of study are partition elements (or path components, or clusters) that {\it persist} through changes of distance parameters, or density parameters, or both --- these objects are called {\it stable components}. The idea
is to assign a more precise meaning to the colloquial version of clusters, such as clusters of stars, which are isolated groupings of objects in a data set. 

To this end, we recall some basic constructions of topological data analysis in the first section of this paper. Specifically, the data set $X$ determines an ascending sequence of simplicial complexes $V_{s}(X)$, the Vietoris-Rips complexes, in which the simplices consist of sets of data points having mutual distance bouned by a non-negative real number $s$. These complexes, in turn, are filtered by Lesnick subcomplexes $L_{s,k}(X)$, where the filtration is determined by a density parameter $k$.

Each of these complexes has a functorial set of path components $\pi_{0}V_{s}(X)$, respectively $\pi_{0}L_{s,k}(X)$, which partition subsets of the data set $X$. The resulting diagrams of partitions have associated hierarchies $\Gamma(V_{\ast}(X))$ and $\Gamma(L_{\ast,\ast}(X))$, containing {\it component graphs} $\Gamma_{0}(\pi_{0}V_{\ast}(X))$ and $\Gamma_{0}(L_{\ast,\ast}(X))$, respectively. The component graphs consist of path components that do not change through some variation in the defining parameters $s$ and/or $k$. The component graphs themselves have path components, and these are the {\it stable components} for the data set $X$, in various incarnations.
\medskip

Stable components further decompose into {\it layers,} for filtrations derived from the Lesnick filtration $L_{s,k}(X)$, in which the underlying sets of vertices may not be constant. The layers of the filtration $L_{\ast,\ast}(X)$  are defined by computing path components of the {\it layer subgraph} $\Gamma_{0}'(L_{s,k}(X))$ of the component graph. The edges of the layer subgraph are defined by path components that have the same {\it size} (or cardinality) through changes of defining parameters. 

Layers are defined and discussed in the second section of this paper. Every stable component is a disjoint union of its constituent layers, while the layers form graphs that are relatively easy to visualize geometrically as unions of squares.
\medskip

Naive calculations of stable components and layers produce large collections of small objects. In particular, the individual elements of $X$ are stable components for the Vietoris-Rips filtration. It is typical to interpret small stable components or layers as noise, and remove them from the output of a particular algorithm. This can be done with a ``scoring'' technique, for which noise objects can be interpreted as stable components having low scores. That said, one could be most interested in stable components having relatively low scores, such as in algorithms that detect money laundering or smaller scale star clusters.

Scoring appears as an analytic device in statistical approaches to clustering --- see \cite{HMc}, for example. An alternative combinatorial method of scoring is presented in the third section of this paper. Basically, the score $\sigma(P)$ of a stable component or the score $\sigma(L)$ of a layer $L$ is the sum of the cardinalities of the path components that appear in its list of vertices. This number is most effectively calculated by making yet another graph out of the vertices of the Lesnick complexes $L_{s,k}(X)$ and computing the cardinality of the set of vertices of a pullback of a stable component or layer within this new graph. This method of scoring is additive, so that the score of a stable component is the sum of the scores of its constituent layers.
\medskip

An simple theoretical problem is presented in the final section of this paper. We start with a data set $X \subset \mathbb{R}^{n}$ and add a point $y$ which is very close to one of the elements of $X$ to form a new data set $Y = X \sqcup \{y\}$, and then we want to compare the stable components for $X$ and $Y$.

Suitably interpreted, the inclusion $ i: X \subset Y$ induces maps of Vietoris-Rips complexes $V_{s}(X) \to V_{s}(Y)$ which are natural with respect to the distance parameter $s$, with a corresponding map of hierarchies
\begin{equation*}
  \Gamma(X) := \Gamma(V_{\ast}(X)) \xrightarrow{i_{\ast}} \Gamma(V_{\ast}(Y)) =: \Gamma(Y).
\end{equation*}

The language of layers is used to compare stable components with respect to this map. In this context, a layer of $X$ can be viewed as an edge
\begin{equation*}
  (s,[x]_{s}) \to (t,[x]_{t})
\end{equation*}
  of $\Gamma(X)$, such that $[x]_{s}=[x]_{t}$ as subsets of $X$ and a maximality condition on the length $t-s$ is satisfied. Here, $[x]_{s}$ is the path component of $x$ in $V_{s}(X)$.

The map $\Gamma(X) \to \Gamma(Y)$ breaks up layers of $X$ having both $y \notin [x]_{s}$ and $y \in [x]_{t}$ in $V_{s}(Y)$ and $V_{t}(Y)$, respectively. Otherwise, layers of $X$ are mapped to ``partial'' layers of $Y$. These partial layers expand to layers of $Y$, having a size that can be approximated. The situation is summarized in Proposition \ref{prop 15} below.

More precise answers to the question of how layers and their scores vary through the hierarchy comparison $i_{\ast}: \Gamma(X) \to \Gamma(Y)$ would be available in specific examples.

The hierarchy $\Gamma(Y)$ is a refinement of $\Gamma(X)$, and the map $i_{\ast}$  is a deformation retract in a very strong sense by the homotopy interleaving methods of \cite{BlumLes}. Statements of this form give coarse information about layers, while the precise layer structure of a data set depends strongly on the distances between its points. Adding a single point to a data set $X$ can introduce many new distances between points of the resulting data set $Y$.
\medskip

This paper was partially conceived and written during a series of visits to the Tutte Institute, and I would like to thank the Insitute for its hospitality and support. This research was also partially supported by the Natural Sciences and Engineering Research Council of Canada.

I would like to thank the referee for raising the question that is the subject of the final section of this paper. This question is quite natural, and the collection of ideas behind the solution that is presented here has proved to be quite provocative.
\medskip

\tableofcontents


\section{Clusters, graphs and stable components}

Clustering is a long-standing enterprise of data analysis, which consists of many definitions and methods. Basically, the standard techniques amount to either partitioning a data set, or constructing hierarchies of partitions.

For example, $K$-means clustering starts with a set of points in a data set. This set of points partitions the data set into regions of nearest neighbours, or Voronoi cells. The algorithm proceeds by finding centres of the cells, and then finding nearest neighbour sets for this set of centres to produce a new set of regions. The procedure stops when the set of centres stabilizes. The Voronoi cells determined by the resulting set of stable points partition the data set, and the partitions are expected to contain dense regions of data points near their centres.

Hierarchical clustering algorithms, such as single linkage clustering, require fewer assumptions. The inductive step in single linkage clustering assumes the existence of a partition
\begin{equation*}
  X = P_{1} \sqcup \dots \sqcup P_{k} \subset \mathbb{R}^{n}
\end{equation*}
of the data set $X$. Find subsets $P_{i}$ and $P_{j}$ which are closest together in the ambient space $\mathbb{R}^{n}$, and then form a coarser partition by taking the union of $Q=P_{i} \cup P_{j}$ while keeping the other partition subsets fixed. There is a corresponding function which relates the first partition to the new one. The algorithm typically starts with the discrete partition $X = \sqcup_{x \in X}\ \{x \}$, and the last possible step in the resulting hierarchy of partitions would be the singleton partition consisting of $X$ alone. The algorithm is typically stopped when it reaches a partition of $X$ with sufficiently many points in some partition members, where the phrase ``sufficiently many'' is open to interpretation. The diagram of partitions and relations between them forms a dendogram, which is a type of tree. 
\medskip

The methods of topological data analysis produce multiple variations of these general themes. 

Suppose that $X \subset \mathbb{R}^{n}$ is a finite data set, and choose a non-negative real number $s$.

The {\it Vietoris-Rips} complex $V_{s}(X)$ is a simplicial complex with simplices consisting of sets $\{x_{0}, \dots ,x_{k}\}$ of points of $X$ such that $d(x_{i},x_{j}) \leq s$ for all $i,j$. The vertices of $V_{s}(X)$ are the elements of $X$, and the set of path components $\pi_{0}V_{s}(X)$ of $V_{s}(X)$ defines a partition of the data set $X$.

Explicitly, elements $x,y \in X$ are in the same path component of $V_{s}(X)$ if there is a series of ``short hops'' (length $\leq s$) from $x$ to $y$ through elements of $X$. We define an equivalence relation on $X$ in this way, and the equivalence classes are the path components of $V_{s}(X)$. Calculation of the set of path components of $V_{s}(X)$ can be done with an algorithm.

If $s \leq  t$ then there is an induced inclusion of simplicial complexes
\begin{equation*}
  V_{s}(X) \subset V_{t}(X),
\end{equation*}
  and a corresponding function $\pi_{0}V_{s}(X) \to \pi_{0}V_{t}(X)$ that relates partitionings given by the respective sets of path components. Because $X$ is finite, there are only finitely many numbers
\begin{equation*}
  0=s_{0} < s_{1} < \dots < s_{p}
\end{equation*}
that I call {\it phase change} numbers, which can occur as distances between elements of $X$. In the corresponding string of inclusions
\begin{equation*}
  X = V_{s_{0}}(X) \subset V_{s_{1}}(X) \subset \dots \subset V_{s_{p}}(X),
\end{equation*}
the complex $V_{s_{0}}(X)$ is the discrete set $X$, while $V_{s_{p}}(X)$ is a big simplex, which I write as $\Delta^{X} =: \Delta^{N}$, where $N+1$ is the number of elements of $X$. The space $V_{s_{p}}(X) = \Delta^{X}$ is contractible.

The contractibility of $V_{s_{p}}(X)$ implies that its set of path components $\pi_{0}V_{s_{p}}(X)$ is a singleton set, and I express this by writing $\pi_{0}V_{s_{p}}(X) = \ast$.

\begin{rem}
One way to produce the complexes $V_{s_{i}}(X)$, altogether, is to first find all distances $s_{i}$ between data points, and to determine the maximum distance between elements for all subsets $\sigma = \{x_{0},x_{1}, \dots ,x_{p}\}$ of $X$. This maximum distance for a subset $\sigma$ is one of the $s_{i}$, and so $\sigma$ is a simplex of $V_{s_{i}}(X)$. This construction is simple enough, but note the exponential complexity \cite{Zomo}.
\end{rem}

The corresponding picture
\begin{equation}\label {eq 1}
  X = \pi_{0}V_{s_{0}}(X) \to \pi_{0}V_{s_{1}}(X) \to \dots \to \pi_{0}V_{s_{p}}(X) = \ast
\end{equation}
of surjective functions between partitions defines a tree by a method specified below (Remark \ref{rem 4}), and as such defines a {\it hierarchical clustering}.

The Vietoris-Rips complex $V_{s}(X)$ can be filtered by density. Suppose that $k$ is a non-negative number. Then the complex $V_{s}(X)$ has a ``full'' subcomplex $L_{s,k}(X)$, which I call a {\it Lesnick complex}, whose simplices consist of vertices having at least $k$ neighbours relative to the parameter $s$. See also \cite{HMc}.

The Lesnick complexes $L_{s,k}(X)$, $k \ge 0$ filter the Vietoris-Rips complex $V_{s}(X)$, and changing either the distance parameter $s$ or the density parameter $k$ defines an array of inclusions
\begin{equation}\label{eq 2}
  \xymatrix@=12pt{
    L_{s,k}(X) \ar[r] & L_{t,k}(X) \\
    L_{s,k+1}(X) \ar[u] \ar[r] & L_{t,k+1}(X) \ar[u]
  }
\end{equation}
of simplicial complexes.

Lesnick says that the array $\{ L_{s,k}(X) \}$ is the degree Rips filtration of the Vietoris-Rips system $\{ V_{s}(X) \}$, \cite{LW2}.

\begin{ex}
The partitioning algorithm DBSCAN* (``Density based spatial clustering of applicatons with noise'') amounts to a calculation of $\pi_{0}L_{s,k}(X)$, for fixed tunable distance parameter $s$ and density parameter $k$. 
\end{ex}

\begin{ex}\label{ex 3}
The ``hierarchical'' version HDBSCAN* \cite{DBLP}, \cite{HMc} of DBSCAN* is the production (and interpretation) of a tree that is associated to the string of functions
\begin{equation}\label{eq 3}
  \pi_{0}L_{s_{0},k}(X) \to \pi_{0}L_{s_{1},k}(X) \to \dots \to \pi_{0}L_{s_{p},k}(X) = \ast.
\end{equation}
Here, $L_{s_{i},k}(X)$ could be empty for small $s_{i}$, and we are assuming that $k$ is bounded above by the cardinality of $X$. There is one tunable parameter in this case, namely the density $k$.
\end{ex}

Starting with the data set $X \subset \mathbb{R}^{n}$ as above,
we continue to examine the string of functions
\begin{equation*}
  X = \pi_{0}V_{s_{0}}(X) \to \pi_{0}V_{s_{1}}(X) \to \dots \to \pi_{0}V_{s_{p}}(X) = \ast
\end{equation*}
as in (\ref{eq 1}), but with a different interpretation.


These functions together determine a graph $\Gamma(X)$, whose vertices are pairs $(s_{i},[x])$, where $[x]$ (sometimes $[x]=[x]_{i}$ or $[x]=[x]_{s_{i}}$, for precision) is a path component of a data point $x \in X$ in the complex $V_{s_{i}}(X)$. The edges of the graph have the form $(s_{i},[x]) \to (s_{i+1},[x])$. The path component represented by $x$ in $V_{s_{i}}(X)$ maps to the path component represented by $x$ in $V_{s_{i+1}}(X)$ under the function $\pi_{0}V_{s_{i}}(X) \to \pi_{0}V_{s_{i+1}}(X)$, as is standard.

The graph $\Gamma(X)$ is the {\it hierarchy graph} for $\pi_{0}V_{\ast}(X)$. This graph is a tree, and is the hierarchical clustering arising from the complexes $V_{s}(X)$, but we go further.

The vertex $(s_{i},[x])$ is a {\it branch point} if, equivalently,
\begin{itemize}
  \item[1)]
    there are distinct path components $[u],[v]$ of $V_{s_{i-1}}(X)$ such that $[u]=[v]=[x]$ in $\pi_{0}V_{s_{i}}(X)$, or
  \item[2)]
    the inclusion $[x]_{i-1} \subset [x]_{i}$ of path components is not surjective.
    \end{itemize}

Remove all edges that terminate in branch points from the hierarchy graph $\Gamma(X)$, to produce a subgraph $\Gamma_{0}(X)$, called the {\it component graph}. The path components of the graph $\Gamma_{0}(X)$ the {\it stable components} of the data set $X$.
\medskip

\begin{rem}\label{rem 4}
These definitions can be generalized to arbitrary strings of functions
\begin{equation*}
  F:\ F_{0} \xrightarrow{\alpha} F_{1} \xrightarrow{\alpha} \dots \xrightarrow{\alpha} F_{p}.
\end{equation*}
The {\it hierarchy graph} $\Gamma(F)$ has vertices $(i,x)$ with $x \in F_{i}$, and has edges
\begin{equation*}
  (i,x) \to (i+1,\alpha(x)).
\end{equation*}
I say that the vertex $(i,x)$ is a {\it branch point} of $\Gamma(F)$ if there are distinct elements $y,z \in F_{i-1}$ such that $\alpha(y) = \alpha(z) = x$. Remove all edges terminating in branch points from the graph $\Gamma(F)$ to form the {\it component graph} $\Gamma_{0}(F)$, and then the {\it stable components} of $F$ are the path components of the graph $\Gamma_{0}(F)$.

At this level of generality, the hierarchy graph $\Gamma(F)$ is a disjoint union of trees, one for each element of the set $F_{p}$.
\end{rem}

\begin{ex}
  The hierarchy graph $\Gamma(\pi_{0}L_{\ast,k}(X))$ that is associated to the string of functions (\ref{eq 3}) is the tree of the HDBSCAN* algorithm.

  The stable components for the string (\ref{eq 3}), meaning the path components of the component graph $\Gamma_{0}(\pi_{0}L_{\ast.k}(X))$, are said to be clusters in \cite[Sec.2.3]{HMc}. 
\end{ex}

Applying the path component construction $\pi_{0}$ to the array of simplicial complexes
\begin{equation*}
  \xymatrix@=12pt{
    L_{s_{i},k}(X) \ar[r] & L_{s_{i+1},k}(X) \\
    L_{s_{i},k+1}(X) \ar[u] \ar[r] & L_{s_{i+1},k+1}(X) \ar[u]
  }
\end{equation*}
produces an array of functions
\begin{equation}\label{eq 4}
\xymatrix@=14pt{
\pi_{0}L_{s_{i},k}(X) \ar[r] & \pi_{0}L_{s_{i+1},k}(X) \\
\pi_{0}L_{s_{i},k+1}(X) \ar[r] \ar[u] & \pi_{0}L_{s_{i+1},k+1}(X) \ar[u]
}
\end{equation}

This array is finite, and it is a special case of a finite array of functions $F$ having the form
\begin{equation}\label{eq 5} 
\xymatrix@=14pt{
F_{0,0} \ar[r]^{\alpha}& F_{1,0} \ar[r] & \dots \\
F_{0,1} \ar[r]_{\alpha} \ar[u]^{\beta} & F_{1,1} \ar[r] \ar[u]_{\beta} & \dots \\
\vdots \ar[u] & \vdots \ar[u]
}
\end{equation}
where $F_{i,k} = \pi_{0}L_{s_{i},k}(X)$.

There is again a ``hierarchy'' graph $\Gamma(F)$ with vertices $((m,k),x)$ with $x \in F_{m,k}$, and edges
\begin{equation*}
  ((m,k),x) \to ((m+1,k),\alpha(x))\enskip \text{and}\enskip
  ((m,k+1),y) \to ((m,k),\beta(y)),
\end{equation*}
called {\it horizontal} and {\it vertical} edges, respectively.

The vertex $((m,k),x)$ is a {\it horizontal branch point} if there are two elements $y,z \in F_{m-1,k}$ such that $\alpha(y)=\alpha(z)=x$. Similarly, $((m,k),x)$ is a {\it vertical branch point} if there are elements $u,v \in F_{m,k+1}$ such that $\beta(u) = \beta(v) =x$.

Remove all edges terminating in {\it either} horizontal or vertical branch points from the graph $\Gamma(F)$ to form the {\it component graph} $\Gamma_{0}(F)$. Observe that the graphs $\Gamma(F)$ and $\Gamma_{0}(F)$ have the same vertices. The path components $\pi_{0}\Gamma_{0}(F)$ are the {\it stable components} of the array $F$. 

\begin{ex}\label{ex 6}
  The set of vertices for the complex $L_{s,k}(X)$ is empty for $k \geq m$, for some $m$. Holding $s$ fixed and letting $k$ vary gives a string of functions
  \begin{equation*}
    \emptyset = \pi_{0}L_{s,m}(X) \to \pi_{0}L_{s,m-1}(X) \to \dots \to \pi_{0}L_{s,0}(X) = \pi_{0}V_{s}(X).
    \end{equation*}
  The corresponding graph $\Gamma(\pi_{0}L_{s,\ast}(X))$ is a disjoint union of trees, which is the analogue of the cluster tree of a density function \cite{Stuetzle}
  for the present context. The cluster tree is also a type of hierarchy graph. See also \cite{CM-2010}.
  \end{ex}

 \section{Layers}

 Suppose again that $X \subset \mathbb{R}^{n}$ is a finite data set, and consider the ascending sequence of Vietoris-Rips complexes
 \begin{equation}\label{eq 6}
   X=V_{s_{0}}(X) \subset V_{s_{1}}(X) \subset \dots \subset V_{s_{p}}(X)
 \end{equation}
 that is associated to the phase change numbers  $0=s_{0} < s_{1} < s_{2} < \dots < s_{p}$. Form the associated string of functions
 \begin{equation*}
X=\pi_{0}V_{s_{0}}(X) \to \pi_{0}V_{s_{2}}(X) \to \dots \to \pi_{0}V_{s_{p}}(X)=\ast
 \end{equation*}
between path components.  

Given an edge $(s_{i},[x]_{s_{i}}) \to (s_{i+1},[x]_{s_{i+1}})$ in the associated hierarchy graph $\Gamma(F)=\Gamma(X)$, the vertex $(s_{i+1},[x]_{s_{i+1}})$ is not a branch point if and only if the cardinalities of the path components $[x]_{s_{i}} \in \pi_{0}V_{s_{i}}(X)$ and $[x]_{s_{i+1}} \in
\pi_{0}V_{s_{i+1}}(X)$ coincide. This means that $[x]_{s_{i}} = [x]_{s_{i+1}}$ as subsets of $X$.

It follows that a stable component for the sequence $F=\{ \pi_{0}V_{s_{j}}(X) \}$
is a string of edges
\begin{equation*}
  P:\ (i,[x]_{i}) \to (i+1,[x]_{i+1}) \to \dots \to (i+n,[x]_{i+n})
\end{equation*}
in $\Gamma(F)$ of maximal length such that all components $[x]_{j} \in \pi_{0}V_{s_{j}}(X)$ coincide as subsets of the data set $X$.

Each stable component $P$ as above contains a unique branch point, namely the vertex $(i,[x]_{i})$ at the beginning of the string. It follows that stable components of $X$ can be identified with the branch points of the hierarchy $\Gamma(F)=\Gamma(X)$.
\medskip

Fatten up the sequence (\ref{eq 6}) of Vietoris-Rips complexes to the array $\{ L_{s_{i},k}(X)\}$ of Lesnick complexes, write
$F_{i,k}= \pi_{0}L_{s_{i},k}(X)$, and form the graph $\Gamma(F)$ as in the last section.

Form a subgraph $\Gamma_{0}'(F)$ of $\Gamma(F)$ by saying that
\begin{equation}\label{eq 7}
  \begin{aligned}
    &((i,k),[x]) \to ((i+1,k),[x])\enskip \text{or}\\
    &((i,k),[z]) \to ((i,k-1),[z])
  \end{aligned}
  \end{equation}
is an edge of $\Gamma_{0}'(F)$ if and only if the path components $[x] \in \pi_{0}L_{i,k}(X)$ and $[x] \in \pi_{0}L_{i+1,k}(X)$ (respectively, $[z] \in \pi_{0}L_{i,k}(X)$ and $[z] \in \pi_{0}L_{i,k-1}(X)$) coincide as subsets of $X$. 

\begin{rem}
If $[z]=[z]_{i,k}$ is path component of $L_{i,k}(X)$, then $z$ represents a path component $[z]_{i,k-1}$ of $L_{i,k-1}(X)$, and there is an inclusion $[z]_{i,k} \subset [z]_{i,k-1}$ of subsets of $X$. If the element $((i,k-1),[z])=((i,k-1),[z]_{i,k-1})$ is a vertical branch point, then $[z]_{i,k-1}$ contains a path component from $L_{i,k}(X)$ in addition to $[z]_{i,k}$, and so $[z]_{i,k} \ne [z]_{i,k-1}$ as subsets of $X$. 

Similar considerations hold in the horizontal case: if $[x]_{i,k}=[x]_{i+1,k}$ as subsets of $X$, then $((i+1,k),[x]_{i+1,k})$ cannot be a horizontal branch point.
\end{rem}

The graph $\Gamma_{0}'(F)$ is the {\it layer graph} for $F$. 
The edges of $\Gamma_{0}'(F)$ are edges of the graph $\Gamma(F)$ for which the size of path component subsets is preserved.

If either of the edges in (\ref{eq 7}) is in the layer subgraph $\Gamma_{0}'(F)$, then the target in each case cannot be a horizontal (respectively) vertical branch point, because of the preservation of size of path components. It follows that the layer graph $\Gamma_{0}'(F)$ is a subgraph of the component graph $\Gamma_{0}(F)$.
\medskip

Suppose that $((s,k),[x])$ is a vertex of $\Gamma_{0}'(F)$, and suppose given a path
\begin{equation*}
  P:\ ((t_{0},k_{0}),[y_{0}]) \to \dots \to ((t_{n},k_{n}),[y_{n}])) = ((s,k),[x])
\end{equation*}
in $\Gamma_{0}'(F)$ which terminates at $((s,k),[x])$. Then $x \in X$ is in all path components $[y_{i}] \in \pi_{0}L_{t_{i},k_{i}}(X)$ since these subsets of $X$ are constant through the path, so that we can rewrite the path $P$ as
\begin{equation*}
  P:\ ((t_{0},k_{0}),[x]) \to \dots \to ((t_{n},k_{n}),[x]) = ((s,k),[x]).
\end{equation*}

In general, the path component $L$ of a vertex $((s,k),[x])$ in $\Gamma_{0}'(F)$ consists of vertices of the form $((t,r),[x])$.

For a vertex $((t,r),[x])$ and a path $P$ as above, if $(t,r)$ satisfies $t_{0} \leq t \leq t_{n}=s$ and $k=k_{n} \leq r \leq k_{0}$, then $x \in L_{t_{0},k_{0}}(X)$ so that $x \in L_{t,r}(X)$, and $x$ represents a path component $[x]$ for the three complexes
\begin{equation*}
  L_{t_{0},k_{0}}(X) \subset L_{t,r}(X) \subset L_{t_{n},k_{n}}(X).
\end{equation*}
It follows that the three path components represented by $x$ coincide, so that $((t,r),[x])$ is in the path component of $((s,k),[x])$.

It follows that every path $P$ in the layer graph $\Gamma'(F)$ generates a {\it square} $\sq(P)$ of vertices $((t,r),[x])$ with $t_{0} \leq t \leq s$ and $k \leq r \leq k_{0}$, and this square lies in the path component of $((s,k),[x])$.

Suppose now that $L$ is a path component in $\Gamma_{0}'(F)$ of a vertex $((t,r),[y])$. Find the elements $((s_{i},k_{i}),[y])$, $i = 1, \dots ,n$, of the component $L$ which are maximal in $s_{i}$ and minimal in $k_{i}$. Find all paths $Q_{1}, \dots ,Q_{n}$ in $L$ which terminate in one of the $((s_{i},k_{i}),[y])$, and which are maximal in the sense that they cannot be extended to longer paths. Then the path component $L$ in $\Gamma_{0}'(F)$ is a union of the squares associated to these maximal paths, in the sense that
\begin{equation}\label{eq 8}
  L = \cup_{i=1}^{n}\ \sq(Q_{i}).
\end{equation}

We have a graph inclusion $\Gamma_{0}'(F) \subset \Gamma_{0}(F)$, and both graphs have the same vertices. It follows that each path component (stable component) $P$ of the component graph $\Gamma_{0}(F)$ is a disjoint union of components of the graph $\Gamma_{0}'(F)$, meaning that
  \begin{equation*}
    P = \sqcup_{j}\ L_{j},
  \end{equation*}
  where the subsets $L_{j}$ are path components of $\Gamma_{0}'(F)$ that are contained in $P$.  In other words, each stable component is a disjoint union of layers. At the same time, we know from (\ref{eq 8}) that each layer $L_{j}$ is a union of squares.

  These observations, taken together, give a geometric picture of the stable components for the Lesnick filtration $\{ L_{s,k}(X) \}$ of a data set $X$.

  \begin{ex}
    The distinction between stable components and layers applied to all subdiagrams of the array $\{ L_{s_{i},k}(X) \}$ for which the complexes involved do not share a common set of vertices.

    Such is the case for the diagram of complexes
    \begin{equation*}
      L_{s_{0},k}(X) \subset L_{s_{1},k}(X) \subset \dots \subset L_{s_{p},k}(X)
    \end{equation*}
    which produces the HDBSCAN* algorithm of Example \ref{ex 3}. The stable components (or clusters of \cite{HMc}) break up into disjoint unions of layers in this case, but this decomposition remains to be interpreted.
    \end{ex}
  
  \section{Scoring}

  Continue with the data set $X \subset \mathbb{R}^{n}$, suppose again that the distinct lengths between elements of the finite set $X$ have the form
  \begin{equation*}
    0=s_{0} < s_{1} < \dots < s_{p}
  \end{equation*}
  and form the array of Lesnick complexes $\{ L_{s_{i},k}(X) \}$. These complexes are subcomplexes of the largest Vietoris-Rips complex $V_{s_{p}}(X)$, which can be identified with a large simplex $\Delta^{N}$, where $X$ has $N+1$ elements.

  Write $K_{s_{i},k}$ for the set of vertices of the complex $L_{s_{i},k}(X)$. Then there is an array of inclusions of vertices $\{ K_{s_{i},k} \}$ and a collection of surjective functions
  \begin{equation}\label{eq 9}
    p:  K_{s_{i},k} \to \pi_{0}L_{s_{i},k}(X).
  \end{equation}
  These functions $p$ are natural in $i$ and $k$, and together define a map of arrays. All sets $K_{s_{i},k}$ are subsets of the data set $X$. The sets $K_{s_{i},0}$ coincide with $X$.

  The function $p$ of (\ref{eq 9}) is defined by $p(x) = [x]$, where $[x]$ is the path component represented by the vertex $x$. The subset $p^{-1}([x])$ of $K_{s_{i},k}$ is the set of members of the path component $[x]$ of the complex $L_{s_{i},k}(X)$.

  The array of sets $K_{s_{i},k}$ defines a graph $\Gamma(K_{\ast,\ast})$ as before. Vertices are pairs $((s_{i},k),y)$ with $y \in K_{s_{i},k}$, and there are horizontal and vertical edges having the respective forms
  \begin{equation*}
    ((s_{i},k),y) \to ((s_{i+1},k),y)\enskip \text{and}\enskip
    ((s_{i},k),y) \to ((s_{i},k-1),y).
  \end{equation*}
  The functions $p$ of (\ref{eq 9}) define a graph homomorphism
  \begin{equation*}
    p: \Gamma(K_{\ast,\ast}) \to \Gamma(\pi_{0}L_{\ast,\ast}(X))
  \end{equation*}
  which is defined on vertices by $p((s_{i},k),y) = ((s_{i},k),[y])$.
  
Given $x \in X$, there is a subgraph $\Gamma_{x}(K) \subset \Gamma(K_{\ast,\ast})$ which has vertices $((s_{j},k),x)$ with $x \in K_{s_{j},k}$. Each subgraph $\Gamma_{x}(K)$ is connected, and the subgraphs $\Gamma_{x}(K)$ are the connected components of $\Gamma(K_{\ast,\ast})$. It follows that there is a graph isomorphism
  \begin{equation*}
    \bigsqcup_{x \in K_{0}}\ \Gamma_{x}(K) \xrightarrow{\cong} \Gamma(K_{\ast,\ast}).
  \end{equation*}

  Let $P$ be a stable component for the diagram $\pi_{0}L_{\ast,\ast}(X)$, and form the pullback diagram
  \begin{equation}\label{eq 10}
    \xymatrix{
    P \cap \Gamma(K_{\ast,\ast}) \ar[r] \ar[d] & \Gamma(K_{\ast,\ast}) \ar[d]^{p} \\
    P \ar[r]_-{i} & \Gamma(\pi_{0}L_{\ast,\ast}(X))
    }
  \end{equation}
  where $i$ is the composite inclusion $P \subset \Gamma_{0}(\pi_{0}L_{\ast,\ast}(X)) \subset \Gamma(\pi_{0}L_{\ast,\ast}(X))$ of graphs. 
  
  There is a graph isomorphism
  \begin{equation*}
  \bigsqcup_{x \in X}\ P \cap \Gamma_{x}(K) \xrightarrow{\cong} P \cap \Gamma(K_{\ast,\ast}) 
  \end{equation*}
  while the set of vertices $(P \cap \Gamma(K_{\ast,\ast})_{0}$ of the graph $P \cap \Gamma(K_{\ast,\ast})$ is a disjoint union of the sets $p^{-1}((s_{i},k),[y])) \subset K(s_{i},k)$ with $((s_{i},k),[y]) \in P$.

  It is standard to write $\vert F \vert$ for the number of elements in a finite set $F$. It follows that there is an identity
  \begin{equation}\label{eq 11}
    \sum_{x \in X}\ \vert (P \cap \Gamma_{x}(K))_{0} \vert =
    \sum_{((s_{i},k),[y]) \in P}\ \vert [y] \vert.
    \end{equation}
  The number
\begin{equation*}
  \zeta(x,P) = \vert (P \cap \Gamma_{x}(K))_{0} \vert
\end{equation*}
  is the number of vertices in $P$ which are represented by $x$. It is a combinatorial {\it stability measure} of $x$ with respect to $P$. This description of $\zeta(x,P)$ is adapted from the stability measure for $x$ of \cite{HMc}, which is an analytic invariant. 
  
The combinatorial persistence {\it score} $\sigma(P)$ of a stable component $P$ is the sum of all stability measures $\zeta(x,P)$:
\begin{equation*}
  \sigma(P) = \sum_{x \in K_{0}}\ \zeta(x,P), 
\end{equation*}
so that
\begin{equation*}
 \sigma(P) = \sum_{x \in X}\ \vert (P \cap \Gamma_{x}(K))_{0} \vert =
    \sum_{((s_{i},k),[y]) \in P}\ \vert [y] \vert.
  \end{equation*}
on account of the identity (\ref{eq 11}). 

Analogs of diagram (\ref{eq 10}) lead to similar analyses for all subobjects of the graph $\Gamma(\pi_{0}L_{\ast,\ast}(X))$, including
\begin{itemize}
\item[1)] layers for $\Gamma(\pi_{0}L_{\ast,\ast}(X))$,
  \item[2)] 
    stable components and layers for the graph $\Gamma(\pi_{0}L_{\ast,k}(X))$ (HDBSCAN* case), and
  \item[3)]
    stable components (which are also layers) for the graph
    $\pi_{0}L_{\ast,0}(X) = \pi_{0}K_{\ast}(X)$.
\end{itemize}

If the subobject $L \subset \Gamma(\pi_{0}L_{\ast,\ast}(X))$ is a layer, then we have a pullback diagram
  \begin{equation*}
    \xymatrix{
    L \cap \Gamma(K_{\ast,\ast}) \ar[r] \ar[d] & \Gamma(K_{\ast,\ast}) \ar[d]^{p} \\
    L \ar[r]_-{i} & \Gamma(\pi_{0}L_{\ast,\ast}(X))
    }
  \end{equation*}
  where $i$ is the inclusion of the layer $L$. Then
  \begin{equation*}
    L \cap \Gamma(K_{\ast,\ast}) = \bigsqcup_{x \in X}\ L \cap \Gamma_{x}(K).
  \end{equation*}
  In this case, either $L \cap \Gamma(K_{\ast,\ast}) = \emptyset$, or the intersection has a vertex $((s_{i},k),x)$ such that $x$ represents every vertex of $L$. It follows that
    \begin{equation*}
      \vert (L \cap \Gamma(K_{\ast,\ast})_{0} \vert = \vert L_{0} \vert \cdot \vert [x] \vert,
    \end{equation*}
    where $x$ is any choice of representative for a vertex $((s_{i},k),[x])$ of $L$, and $L_{0}$ is the set of vertices of the layer $L$. This is consistent with counting cardinalites of the fibres $p^{-1}(s_{i},k),[x])$, as dictated by the right hand side of equation (\ref{eq 11}).

    The score $\sigma(L)$ of a layer $L$ then has the rather simple form
    \begin{equation*}
      \sigma(L)
= \sum_{x \in X}\ \vert (L \cap \Gamma_{x}(K))_{0} \vert =
    \sum_{((s_{i},k),[y]) \in L}\ \vert [y] \vert = \vert L_{0} \vert \cdot \vert [x]\vert.
    \end{equation*}

    Note finally that since a stable component $P$
    is a disjoint union
    \begin{equation*}
      P = L_{1} \sqcup \dots \sqcup L_{k}
    \end{equation*}
    and scoring for $P$ and $L$ amounts to counting fibres for the graph map $p$, then there is a relation
    \begin{equation*}
      \sigma(P) = \sum_{i=1}^{k}\ \sigma(L_{i})
    \end{equation*}
    which relates the score of $P$ to the scores of its constituent layers.
    
    \section{Adding a point}

    Suppose that $Y = X \sqcup \{ y \}$ and that $d(y,x_{0})< r$ for some $x_{0} \in X$. The number $r$ should be tiny --- this will be made precise later.

    In this section, we compare stable components for the hierarchies $\Gamma(X) = \Gamma(V_{\ast}(X))$ and $\Gamma(Y) = \Gamma (V_{\ast}(Y))$ arising from the respective Vietoris-Rips complexes.
    
    The comparison requires a more versatile approach: 
we expand the systems $V_{\ast}(X)$ and $V_{\ast}(Y)$ (that are indexed by phase change numbers in previous sections) to functors
\begin{equation*}
  s \mapsto V_{s}(X)\enskip \text{and}\enskip s \mapsto V_{s}(Y),
\end{equation*}
respectively, for numbers $s$ in an interval $[0,R]$, where $R$ is an upper bound for the phase change numbers of both $X$ and $Y$.

We are then entitled to simplicial set maps $i:V_{s}(X) \subset V_{s}(Y)$ that are natural in $s \in [0,R]$, which are induced by the inclusion $X \subset Y$. The interesting behaviour of both of these systems, in isolation, occurs at the respective phase shift numbers, but a robust comparison mechanism requires all parameters.

    The functor $s \mapsto V_{s}(X)$ defines a set-valued functor $s \mapsto \pi_{0}V_{s}(X)$, and an expanded hierarchy $\Gamma(X)$ can be defined by analogy with what we have above: $\Gamma(X)$ is a graph (or a category) with vertices $(s,[x]_{s})$ with $s \in [0,R]$ and $[x]_{s} \in \pi_{0}V_{s}(X)$, and has edges $(s,[x]_{s}) \to (t,[x]_{t})$ with $s \leq t$. We shall sometimes write $[x]_{s,X}$ for $[x]_{s}$ where comparisons of data sets are involved.

    An inclusion of data sets $X \subset Y$ then induces a morphism of graphs
\begin{equation*}
  i:\Gamma(X) \to \Gamma(Y).
\end{equation*}

    For such comparisons, it is convenient to phrase the discussion in terms of layers. The layers of $\Gamma(X)$ can be identified with edges $L:(s,[x]_{s}) \to (t,[x]_{t})$ of $\Gamma(X)$ such that
    \begin{itemize}
      \item[1)]
        $s$ and $t$ are phase shift numbers of $X$,
\item[2)] 
  $[x]_{s}=[x]_{t}$ as subsets of $X$, and
\item[3)]
  the {\it length} $t-s$ of the edge $L$ is maximal with respect to the first two conditions.
\end{itemize}

    The vertex $(s,[x]_{s})$ at the initial vertex of a layer must be a branch point, and the full collection of layers of $\Gamma(X)$ can be identified with it set $Br(X)$ of branch points. The layers of $\Gamma(X)$ are the path components (the stable components) of the component graph $\Gamma_{0}(X)$ of the first section.

    A {\it partial layer} for $X$ is an edge $(s,[x]_{s}) \to (t,[x]_{t})$ of $\Gamma(X)$ such that $[x]_{s} = [x]_{t}$. In other words, a partial layer is an edge which satisfies only condition 2) above.
    \medskip
    
 Returning to the situation $X \subset Y = X \sqcup \{y\}$ of interest, write
  \begin{equation*}
    0=s_{0} < s_{1} < \dots < s_{k}
  \end{equation*}
for the list of phase change numbers (i.e. distances between points) for $X$.
  
We suppose that $y$ is close to an element $x_{0} \in X$, in the sense that $d(y,x_{0}) < r$, where
\begin{equation}\label{eq 12}
  r < s_{i+1}-s_{i}
\end{equation}
  for all $i \geq 0$. The inequality (\ref{eq 12}) is the meaning of the requirement that the number $r$ should be tiny.
  
  \begin{lemma}\label{lem 9}
    Suppose that $x \in X$ has the property that $y$ is not a member of the path component $[x]_{s,Y}$ of $V_{s}(Y)$.

    Then the inclusion $[x]_{s,X} \subset [x]_{s,Y}$ is a bijection.
  \end{lemma}

  \begin{proof}
    Suppose that $x' \in [x]_{s,Y}$. Then $y \notin [x]_{s,Y}$, so there is a path
    \begin{equation*}
      x = y_{0} \leftrightarrow y_{1} \leftrightarrow \dots \leftrightarrow y_{r} = x'
    \end{equation*}
    such that all $y_{i} \in X$. It follows that $x' \in [x]_{s,X}$.
    \end{proof}

  \begin{corollary}\label{cor 10}
    Suppose that $ s \leq t$ and that $y \notin [x]_{t,Y}$ in $V_{t}(Y)$.

    If $(s,[x]_{s,X}) \to (t,[x]_{t,X})$ is a partial layer for $X$ then $(s,[x]_{s,Y}) \to (t,[x]_{t,Y})$ is a partial layer for $Y$.
  \end{corollary}

  Note that if $y \notin [x]_{t,Y}$ then $y \notin [x]_{s,Y}$, because $[x]_{s,Y} \subset [x]_{t,Y}$ for $s \leq t$.

  \begin{proof}[Proof of Corollary \ref{cor 10}]
There is a commutative diagram of inclusions
\begin{equation*}
  \xymatrix{
    [x]_{s,X} \ar[r]^{=} \ar[d]_{=} & [x]_{t,X} \ar[d]^{=} \\
    [x]_{s,Y} \ar[r] & [x]_{t,Y}
  }
\end{equation*}
in which the vertical maps are bijections by Lemma \ref{lem 9}. The claim follows.
    \end{proof}

  \noindent
      {\bf Remark}:\enskip There could be layers $(s,[x]_{s,X}) \to (t,[x]_{t,X})$ of $X$ such that $y \notin [x]_{s,Y}$ and $y \in [x]_{t,Y}$. In that case, the corresponding edge $(s,[x]_{s,Y}) \to (t,[x]_{t,Y})$ cannot be a partial layer of $Y$ because the set $[x]_{t,Y}$ is strictly larger than $[x]_{s,Y}$.
      \medskip

In general, if $L$ is a simplicial complex and $V$ is a set of vertices of $L$, the {\it full subcomplex} of $L$ on the set $V$ consists of those simplices $\sigma$ of $L$ having all of their vertices in the set $V$.
      
  Write $V_{s}(X)(y)$ for the full subcomplex of $V_{s}(X)$, on the set of those vertices $x$ such that $[x]_{s,Y}=[y]_{s,Y}$ in $V_{s}(Y)$. 

  \begin{lemma}\label{lem 11}
    Suppose that $s_{i} \leq s < s_{i+1}-r$ for all $i \geq 1$. Then we have the following:
    \begin{itemize}
    \item[1)]
      The space $V_{s}(X)(y)$ is connected.
    \item[2)]
      The function $\pi_{0}V_{s}(X) \to \pi_{0}V_{s}(Y)$ is a bijection.
    \end{itemize}
  \end{lemma}
  
  \begin{proof}
  1)\enskip  Suppose that $x$ is a vertex of $V_{s}(X)(y)$. Then there is a sequence of $1$-simplices
    \begin{equation*}
      x=y_{0} \leftrightarrow y_{1} \leftrightarrow \dots \leftrightarrow y_{n} \leftrightarrow y
      \end{equation*}
    in $V_{s}(Y)$  with $y_{i} \in X$. Then $d(y,x_{0}) < r$, so $d(y_{n},x_{0}) < s+r$ and $[x]=[x_{0}]$ in $V_{s+r}(X)$. But $V_{s}(X) = V_{s+r}(X)$ since $s_{i} \leq s < s+r < s_{i+1}$, and so $[x]=[x_{0}]$ in $V_{s}(X)$. This is true for all vertices $x \in V_{s}(X)(y)$, so that $V_{s}(X)(y)$ is connected.
    \smallskip

    \noindent
    2)\enskip Suppose that $x_{1},x_{2}$ are vertices of $V_{s}(X)$ such that $[y] \ne [x_{1}]$ and $[y] \ne [x_{2}]$ in $V_{s}(Y)$. Then it follows from Lemma \ref{lem 9} that $[x_{1}]_{s,X} = [x_{2}]_{s,X}$ in $V_{s}(Y)$ if and only if $[x_{1}]_{s,Y} = [x_{2}]_{s,Y}$ in $V_{s}(X)$.

    The function $i_{\ast}: \pi_{0}V_{s}(X) \to \pi_{0}V_{s}(Y)$ is surjective since $d(y,x_{0}) < r < s_{1} \leq s$, so that $[y]_{s,Y} = [x_{0}]_{s,Y}$. The previous paragraphs imply that $i_{\ast}$ is injective.
  \end{proof}

  \begin{corollary}
Under the assumptions for Lemma \ref{lem 11}, the map $i_{\ast}: \pi_{0}V_{s_{i}}(X) \to \pi_{0}V_{s_{i}}(Y)$ is a bijection at all phase change numbers $s_{i}$ for $X$,  for $i \geq 1$.
  \end{corollary}

\begin{corollary}\label{cor 13}
  Suppose that the edge $(s,[x]_{s,X}) \to (t,[x]_{t,X})$ is a partial layer of $X$. Suppose that $y \in [x]_{s,Y}$ and that
  $s$ and $t$ are phase change numbers for $X$, with $s_{1} \leq s \leq t$.

  Then the edge $(s,[x]_{s,Y}) \to (t,[x]_{t,Y})$ of $\Gamma(Y)$ is a partial layer of $Y$.
\end{corollary}

\begin{proof}
The spaces $V_{s}(X)(y)$ and $V_{t}(X)(y)$ are connected by Lemma \ref{lem 11}. It follows that
\begin{equation*}
\begin{aligned}
{}\enskip  [x]_{s,Y} &= \{y\} \sqcup (V_{s}(X)(y))_{0} = \{y\} \sqcup [x]_{s,X} \\
&= \{y\} \sqcup [x]_{t,X} = \{y\} \sqcup (V_{t}(Y)(t))_{0} = [x]_{t,Y}
\end{aligned}
\end{equation*}
in $Y$.
\end{proof}

\begin{corollary}
  Suppose that $(s,[x]_{s,Y}) \to (t,[x]_{t,Y})$ is a partial layer of $Y$ such that $y \in [x]_{s,Y}$ and $x \ne y$. Suppose that $s, t$ are phase change numbers for $X$ with $s_{1} \leq s \leq t$.

  Then the edge $(s,[x]_{s,X}) \to (t,[x]_{t,X})$ is a partial layer of $X$.
\end{corollary}

\begin{proof}
The function
$V_{s}(X)(y)_{0} \to V_{t}(X)(y)_{0}$ is a bijection, since $[x]_{s,Y} \to [x]_{t,Y}$ is a bijection, and the spaces $V_{s}(X)(y)$ and $V_{t}(X)(y)$ are path connected by Lemma \ref{lem 11}. It follows that the function $[x]_{s,X} \to [x]_{t,X}$ is a bijection.
\end{proof}

Suppose that $L: (s,[x]_{s,X}) \to (t,[x]_{t,X})$ is a layer of $X$ such that either $y \in [x]_{s,Y}$ or $y \notin [x]_{t,Y}$ and $s_{1} \leq s$. In both cases, the edge $(s,[x]_{s,Y}) \to (t,[x]_{t,Y})$ is a partial layer for $Y$, by Corollary \ref{cor 13} and Corollary \ref{cor 10}, respectively.

Suppose that $s'$ and $t'$ are phase change numbers of $X$ such that
\begin{equation*}
  s_{1} \leq s' < s \leq t < t'
\end{equation*}
and such that $(s',[x]_{s',Y}) \to (t',[x]_{t',Y})$ is a partial layer of $Y$. Then we have the following:
\medskip

\noindent
1)\enskip If $y \in [x]_{s,Y}$ then $y \in [x]_{s',Y}$. The maps $[x]_{s',X} \to [x]_{s,X}$ and $[x]_{t,X} \to [x]_{t',X}$ are bijections on account of the connectedness assertion of Lemma \ref{lem 11}, and it follows that $(s',[x]_{s',X}) \to (t',[x]_{t',X})$ is a partial layer of $X$. But $L$ is a layer, so that $s'=s$ and $t=t'$ by the maximality condition.
\medskip

\noindent
2)\enskip Suppose that $y \notin [x]_{t,Y}$. Then $y \notin [x]_{t',Y}$ since $[x]_{t,Y} = [x]_{t',Y}$. Then the edge $(s',[x]_{s',X}) \to (t',[x]_{t',X})$ is a partial layer of $X$, since the diagram of functions
\begin{equation*}
  \xymatrix{
    [x]_{s',X} \ar[r] \ar[d]_{=} & [x]_{t',X} \ar[d]^{=} \\
    [x]_{s',Y} \ar[r]_{=} & [x]_{t',Y}
  }
\end{equation*}
commutes, as in the proof of Corollary \ref{cor 10}. The maximality condition for the layer $L$ again implies that $s'=s$ and $t=t'$.

We have proved the following result:

\begin{proposition}\label{prop 15}
  Suppose that $L: (s_{i},[x]_{s_{i},X}) \to (s_{j},[x]_{s_{j},X})$ is a layer of $X$ with $s_{1} \leq s_{i}$. Suppose that either $y \in [x]_{s_{i},Y}$ or $y \notin [x]_{s_{j},Y}$.

  Then the edge $L: (s_{i},[x]_{s_{i},Y}) \to (s_{j},[x]_{s_{j},Y})$ is a partial layer for $Y$, and the layer $(s,[x]_{s,Y}) \to (t,[x]_{t,Y})$ containing $L$ must have $s_{i-1} < s \leq s_{i}$ and $s_{j} \leq t < s_{j+1}$.
\end{proposition}

The numbers $s$ and $t$ in the statement of Proposition \ref{prop 15} are phase change numbers for $Y$. One can show that the phase change numbers for $Y$ lie in the open intervals $(s_{i}-r,s_{i}+r)$ centred at the phase change numbers $s_{i}$ for $X$. This is a constraint on the position of the numbers $s$ and $t$ above, but the differences $s_{i}-s_{i-1}$ and $s_{j+1}-s_{j}$ could be large, and there is no theoretical information, for example, on whether $s$ is close to $s_{i}$ or close to $s_{i-1}$.

\bibliographystyle{plain} 
\bibliography{spt}

\end{document}